\documentclass[11pt,a4paper]{article}
\usepackage{fancyhdr}
\usepackage{amsmath}
\usepackage{amsthm}
\usepackage{amssymb}
\usepackage{wasysym}
\usepackage{paralist}
\usepackage{makeidx}
\usepackage[all]{xy}
\usepackage{mathdots}
\usepackage{yhmath}
\usepackage{mathtools}

\usepackage{young}
\usepackage[vcentermath]{youngtab}

\usepackage{graphicx}
\usepackage{epstopdf}

\input xy

\newcommand{\nc}{\newcommand}
 
 \nc{\cl}{\centerline}
 
 \nc{\rk}{{\rm rk}}
 
  \nc{\Ker}{{\rm Ker }}
  
    \nc{\Tab}{{\rm Tab }}
    
    \nc{\Stan}{{\rm Stan}}
    
        \nc{\AStan}{{\rm AStan}}

 \nc{\g}{{\frak g}}

 \nc{\SL}{{\rm SL}}
 
 \nc{\lub}{{\rm lub}}
 
 \nc{\sh}{{\rm sh}}
 \nc{\hatQ}{{\hat Q}}
 \nc{\sgn}{{\rm sgn}}
 \nc{\ad}{{\rm ad}}
 
 \nc{\Cent}{{\rm Cent}}
 
 \nc{\len}{{\rm len}}
 
  \nc{\idx}{{\rm index}}
 
 \nc{\Mat}{{\rm Mat}}
 
  \nc{\Ann}{{\rm Ann}}
  
  \nc{\A}{{\mathcal A}}
  
    \nc{\B}{{\mathcal B}}
    
      \nc{\C}{{\mathcal C}}
      
           \nc{\DD}{{\mathcal D}}
 
 \nc{\Loewy}{{\rm Loewy}}
 
 \nc{\Supp}{{\rm Supp}}
 
 \nc{\env}{{\rm env}}

 \nc{\wt}{{\rm wt}}
 
 \nc{\poll}{{\rm \l}}
 
 \nc{\efp}{{\Bbb F}_p}

 \newcommand{\Par}{{\rm Par}}

\nc{\baru}{{\overline u}}

\nc{\baralpha}{{\overline \alpha}}

\nc{\bargamma}{{\overline \gamma}}

\nc{\barA}{{\bar A}}

\nc{\barq}{{\overline q}}

 \nc{\Orb}{{\rm Orb}}

 \nc{\hatsigma}{{\hat \sigma}}

 \nc{\hatpi}{{\hat \pi}}
 
  \nc{\hatzeta}{{\hat \zeta}}

 \nc{\Ocal}{{\mathcal O}}

 \nc{\M}{\mathfrak{m}}
 
 \nc{\seee}{\mathbb C}
 
 \nc{\varleq}{\preccurlyeq}

 \nc{\hatlambda}{{\hat\lambda}}
 
 \nc{\hatphi}{{\hat \phi}}
 
 \nc{\daggerlambda}{{\lambda^\dagger}}

    \nc{\barr}{{\bar r}}
  \nc{\bart}{{\bar t}}
  
    \nc{\barsigma}{{\bar \sigma}}

  \newcommand{\resp}{{resp.\,}}

\nc\diag{{\rm diag}}
\renewcommand{\vert}{{\,|\,}}
\nc{\hatL}{{\hat L}}

\nc{\barE}{{\bar   E}}
\nc{\D}{{\mathcal D}}
\nc{\E}{{\mathcal E}}
\nc{\F}{{\mathcal F}}
\nc{\FF}{{\mathcal F}}
\nc{\I}{{\mathcal I}}
\nc{\even}{{\rm e}}
\nc{\ep}{\epsilon}
\nc{\odd}{{\rm o}}
\nc{\Coker}{{\rm Coker}}
\nc{\olE}{{\overline E}}
\nc{\indBG}{{\rm ind}_B^G\,}
\nc{\indHG}{{\rm ind}_H^G\,}

\nc{\que}{{\mathbb Q}}
\nc{\barlambda}{{\bar\lambda}}
\nc{\barmu}{{\bar\mu}}
\nc{\barnu}{{\bar\nu}}
\nc{\bartau}{{\bar\tau}}
\nc{\barm}{{\overline  m}}
\nc{\divind}{{\rm div.ind}}
\nc{\tl}{{\tilde{\lambda}}}
\nc{\dar}{\downarrow}

\nc{\en}{\mathbb N}

\nc{\eno}{{\mathbb N}_0}

\nc{\Sym}{{\rm Sym}}
\nc{\Symm}{{\rm Sym}}

\newcommand{\q}{\quad}
\newcommand{\de}{\delta}

\newcommand{\Mod}{{\rm Mod}}
\renewcommand{\mod}{{\rm mod}}

\newcommand{\Sp}{{\rm Sp}}

\newcommand{\bs}{\bigskip}

\renewcommand{\vert}{\,|\,}

\renewcommand{\sgn}{{\rm sgn}}

\xyoption{all}

\renewcommand{\vert}{\,|\,}
 \newcommand{\tbw}{\textstyle\bigwedge}
\newcommand{\zed}{{\mathbb Z}}

\newcommand{\End}{{\rm End}}
\newcommand{\Hom}{{\rm Hom}}
\newcommand{\cf}{{\rm cf}}

\renewcommand{\mod}{{\rm mod}}

\newcommand{\GL}{{\rm GL}}

\nc{\barB}{{\overline B}}
\nc{\barb}{{\overline b}}

\renewcommand{\mod}{{\rm{mod}}}

\nc{\geom}{{\rm geom}}
\nc{\rep}{{\rm rep}}

\newtheorem{definition}{Definition}[section]
\newtheorem{proposition}[definition]{Proposition}
\newtheorem{theorem}[definition]{Theorem}
\newtheorem{lemma}[definition]{Lemma}
\newtheorem{corollary}[definition]{Corollary}
\newtheorem{example}[definition]{Example}

\newtheorem{remark}[definition]{Remark}

\begin{document}


\centerline{\bf  Greene's Theorem and ideals of the group algebra }

\centerline{\bf of a symmetric group}

\bigskip

\centerline{Stephen Donkin}

\bigskip

{\it Department of Mathematics, University of York, York YO10 5DD}

\medskip

\centerline{\tt stephen.donkin@york.ac.uk}

\bigskip

\centerline{2 December     2024}

\bs\bs

\centerline{\bf Abstract}

\q We show that certain factor rings of the group algebra of a symmetric group have natural bases of group elements. These include the factor rings studied by  Raghavan, Samuel and Subrahmanyam, \cite{RSS} and by Doty, \cite{DoDS}. 
We also give generators for the annihilator of certain permutation modules for symmetric groups.

\bs\bs
\section{Introduction}

\q The problem of describing  explicit bases of certain factor rings of the group algebra of a symmetric group has been considered in \cite{RSS},  \cite{DoDS}, \cite{BDMCanBases}.  

\q For a positive integer $n$ we write $[1,n]$ for the set $\{1,\ldots,n\}$ and $\Sym(n)$  the group of permutations of $[1,n]$.  For positive integers $n,r$ we have the set $I(n,r)$ of mappings from $[1,r]$ to $[1,n]$.  Then $\Sym(n)$ and $\Sym(r)$ act on $I(n,r)$ by composition of mappings. Thus $\sigma i= i\circ \sigma$ and $\tau i= i\circ \tau ^{-1}$, for $\sigma\in \Sym(n)$, $\tau\in \Sym(r)$, $i\in I(n,r)$.  We fix a commutative ring $R$.  Then  the $R$-module with $R$-basis $I(n,r)$ is  an  $R\Sym(n)$ permutation module and an  $R\Sym(r)$ permutation module.

\q  It was shown by Raghavan, Samuel and Subrahmanyam, \cite[Theorem 1]{RSS}, that the group algebra $R \Sym(r)$  modulo the annihilator  of the action has  an $R$-basis consisting of the permutations that  contain no decreasing sequence of length more that $n$.   Similarly it is shown \cite[Theorem 1]{DoDS} that $R\Sym(n)$ modulo the annihilator of the action has an $R$-basis consisting of those permutations that  contain an increasing sequence of length at least $n-r$.

\q These results  are  closely related to the theorem of Schensted describing the number of rows  of the shape of the Robinson-Schensted correspondent of a  permutation in terms of increasing or decreasing sequences, \cite[Theorem 2]{Schensted}.  This was generalised  in the theorem of C. Greene, \cite[Section 1, Theorem]{Greene}, to give   a similar description of the entire shape of the Robinson-Schensted correspondent of an arbitrary permutation.  Our purpose here is give a general   result on bases of certain  factor rings of the group algebra of a symmetric group  corresponding to  Greene's description of the shape of the Robinson-Schensted correspondent of an arbitrary permutation, from which the cited theorems of  Raghavan, Samuel and Subrahmanyam and of Doty may be obtained as special cases. 

\q The earlier results rely on the cell structure on Hecke algebras of type $A$ due to Kazhdan and Lusztig, \cite{KL}, and the structure due to Murphy,  \cite{Murphy} (and the relationship between these structures as described by Geck, \cite{Geck}).  
  Our arguments are elementary.
We give these in Section 2.  

\q Section 3 is less elementary and using, in particular, results from \cite{DonkinPermMods}, we give a description of the annihilator of certain permutation modules for symmetric groups   in a more general setting.
\bs\bs\bs

\section{Ideals and bases}

\q We fix a positive integer $n$.  We write  $\Par(n)$ for the set of partitions of $n$ and $S(n)$ for  the set of sequences of length $n$  of distinct elements of $\{1,\ldots,n\}$.   We shall also write a sequence $(s_1,\ldots, s_n)$ in $S(n)$ as the  word $s_1\ldots s_n$.       We write $\Sym(X)$ for the group of permutations of a finite set $X$.  If $[1,n]$ is a the disjoint union of non-empty subsets $X_1,\ldots,X_m$ then we regard $\Sym(X_1)\times \cdots \times \Sym(X_m)$ as a subgroup of $\Sym(n)$ in the natural way. For $s=s_1\ldots s_n\in S(n)$ we write $\pi(s)$ for the corresponding element of $\Sym(n)$, written in two line notation as 
$$\left( \begin{matrix}1 & 2 & \cdot  &\cdot & \cdot & n\\
s_1 & s_2 &\cdot &\cdot &\cdot & s_n
\end{matrix} \right).$$

\q We have an action of  $\Sym(n)$ on $S(n)$. For $s\in S(n)$, $\sigma\in \Sym(n)$, we have $(s \sigma)_i=s_{\sigma(i)}$, $1\leq i\leq n$. 
The  action is simply transitive.  Putting $\alpha=12\ldots n\in S(n)$, we have $\alpha \pi(s)=s$, for $s\in S(n)$.

\q To $s\in S(n)$ there corresponds, via the Robinson-Schensted correspondence,  a pair $(S,T)$ of standard $\lambda$-tableaux with entries in $[1,n]$ for some $\lambda\in \Par(n)$. We call $\lambda$ the \emph{shape} of $s$ and denote it $\sh(s)$. 

\q We say that $s$ is \emph{ascending}  or \emph{increasing}   on a subset $X$ of $[1,n]$ if for all $i,j\in X$ with $i<j$ we have $s_i<s_j$. 

\q We use the dominance partial order $\trianglelefteq$  on $\Par(n)$ and the lexicographic (total) order on $S(n)$.  We shall write $\lub(\Gamma)$ for the least upper bound of a subset $\Gamma$ of $\Par(n)$.

\begin{lemma} Suppose $s\in S(n)$ and $[1,n]$ is the disjoint union of non-empty subsets $X_1,\ldots,X_m$, with $s$ increasing  on $X_r$, $1\leq r\leq m$. Then we have  $s\sigma > s$  for all $1\neq \sigma \in \Sym(X_1)\times \cdots \times \Sym(X_m)$. 
\end{lemma}

\begin{proof}  Suppose for a contradiction that $s\sigma <  s$ and let $k$ be minimal such that $(s\sigma)_k\neq s_k$, so that $(s\sigma)_k<  s_k$, i.e., $s_{\sigma(k)} < s_k$.  Now $k\in X_r$, for some $1\leq r\leq m$ and $s$ is increasing  on $X_r$ so that $\sigma(k)< k$. But then $s_{\sigma(k)}=s_{\sigma(\sigma(k))}$ (by the minimality of $k$) so that $\sigma(k)=\sigma(\sigma(k))$ and $k=\sigma(k)$. But $s_{\sigma(k)}< s_k$, so this is impossible.
\end{proof}



\q  A partition  $\mu=(\mu_1,\ldots,\mu_m)\in \Par(n)$ will be said to be \emph{upwardly}   compatible with $s\in S(n)$ if  
$[1,n]$ may be expressed as a disjoint union of non-empty subsets $X_1,\ldots, X_m$ of sizes $\mu_1,\ldots,\mu_m$ such that $s$ is ascending    on $X_r$, for $1\leq r\leq m$. 

\q For $s\in S(n)$ we write $A(s)$  for the set of all $\mu\in \Par(n)$ upwardly   compatible  with $s$.

\q The theorem of Greene, \cite[Section 1, Theorem]{Greene},  is then:

\begin{theorem} (C. Greene) For  $s\in S(n)$  we have $\sh(s)=\lub(A(s)).$
\end{theorem}

\begin{example} Let $s=412563$. Then $A(s)$ contains $(4,1,1)$ and $(3,3)$ but not $(4,2)$. Thus $\sh(s)=(4,2)$, and $\sh(s)\not \in A(s)$. 
\end{example}

\q For $s=s_1\ldots s_n\in S(n)$ we write  $s^*$ for the reverse sequence $s_n\ldots s_1$.    For   $\lambda \in \Par(n)$ we write $\lambda^*$  for the dual (or transpose) partition.  For a subset $\Gamma$ of $\Par(n)$ we write $\Gamma^*$ for $\{\lambda^* \vert  \lambda\in \Gamma\}$. 

\q We have $\sh(s^*)=\sh(s)^*$, for $s\in S(n)$, by a result of Schensted,  \cite[Lemma 7]{Schensted}.

\q For $\Gamma\subseteq \Par(n)$ we define 
$$G'(\Gamma)=\{s \vert s\in S(n),  \sh(s)\in \Gamma\}$$
and
$$G(\Gamma)=\{\pi(s) \vert s\in G'(\Gamma) \}.$$

\q For $\mu=(\mu_1,\ldots,\mu_m)\in \Par(n)$ we will write $\Sym(\mu)$ for $\Sym(X_1)\times \cdots \times \Sym(X_m)$, where $X_1=\{1,2,\ldots,\mu_1\}$, $X_2=\{\mu_1+1,\ldots,\mu_1+\mu_2\}$ and so on.

\q  Let $R$ be a commutative ring. For a subset $H$ of $\Sym(n)$ we write $[H]$ for the element $\sum_{h\in H} h$ of $R\Sym(n)$.  Note that the ideal of $R \Sym(n)$ generated by $[\Sym(\mu)]$ is the ideal generated by all $[\Sym(X_1)\times \cdots \times  \Sym(X_m)]$, where $X_1,\ldots,X_m$ are disjoint subsets of $[1,n]$ of sizes $\mu_1,\ldots, \mu_m$.

\q For a subset $\Delta$ of $\Par(n)$ we write $I_R(\Delta)$ for the ideal of $R\Sym(n)$ generated by $[\Sym(\lambda)]$, $\lambda\in \Delta$. We just write $I(\Delta)$ for $I_\zed(\Delta)$.

\q We call a subset $\Delta$ of a partially ordered set $P$ \emph{saturated}  (\resp {\emph {co-saturated}) if whenever $\mu\in \Delta$ $\lambda \in P$ and $\lambda\leq  \mu$ (\resp $\lambda \geq   \mu$) then $\lambda\in \Delta$.  
Note that $\Delta$ is saturated if and only if the complement $P \backslash \Delta$ is co-saturated.   In this section $P$ will be $\Par(n)$ with the dominance partial order $\trianglelefteq$.

\q For $\Delta\subseteq P=\Par(n)$ we define $M(\Delta)=\zed G(P\backslash \Delta) + I(\Delta)\subseteq \zed \Sym(n)$ and  $M'(\Delta)=\zed G'(P\backslash \Delta)+ \alpha I(\Delta)\subseteq \zed S(n)$.

\begin{lemma} Let $\Delta$ be a co-saturated subset of $\Par(n)$. Suppose  $M'(\Delta)\neq \zed S(n)$ and  $s\in S(n)$ is the largest element not in $M'(\Delta)$. Then $A(s)\bigcap \Delta=\emptyset$.
\end{lemma}

\begin{proof}  If $A(s)\bigcap \Delta=\{(1^n)\}$ then $1\in I(\Delta)$ so that $I(\Delta)=\zed \Sym(n)$, and hence $M'(\Delta)=\zed S(n))$. 
 If $A(s)\bigcap \Delta$ is non-empty we may thus pick $(1^n)\neq \mu\in A(s)\bigcap \Delta$.  We  write $[1,n]$ as the disjoint union of subsets $X_1,\ldots,X_m$  of sizes $\mu_1,\ldots,\mu_m$ with $s$ increasing  on each $X_i$.   Thus we have 
   $$s[\Sym(X_1)\times \cdots \times \Sym(X_m)]= \alpha \pi(s) [\Sym(X_1)\times \cdots \times \Sym(X_m)] \in M'(\Delta)$$
   i.e.,
   $$s+\sum_{1\neq \sigma\in \Sym(X_1)\times \Sym(X_2)\times\cdots}  s\sigma\in M'(\Delta).$$
   But $s\sigma>s$ and hence $s\sigma\in M'(\Delta)$,  for $1\neq \sigma\in \Sym(X_1)\times \Sym(X_2)\times\cdots$. Hence $s\in M'(\Delta)$, a contradiction.   Thus $A(s)\bigcap \Delta=\emptyset$.
 \end{proof}

\begin{lemma} Let $\Delta$ be a co-saturated subset of $P=\Par(n)$ such that $P\backslash \Delta$ has a unique maximal element.   Then 
$$R\Sym(n)=RG(P\backslash \Delta)+  I_R(\Delta).$$
\end{lemma}

   \begin{proof}   We may assume $R=\zed$.   By transport of structure it is enough to prove that $M'(\Delta)=\zed S(n)$.  Suppose not and let $s$ be the largest  element of $S(n)$ not in $M'(\Delta)$. Then certainly $s\not \in G'(P\backslash \Delta)$, i.e., $\lambda=\sh(s)\in \Delta$.   If $A(s)\subseteq P\backslash \Delta$ then, since $P\backslash \Delta$ has a unique maximal element, $\lambda=\lub(A(s))\in P\backslash \Delta$, a contradiction.   But then  $A(s)\bigcap \Delta\neq \emptyset$, contrary to Lemma 2.4.
    \end{proof}

\q For $\lambda\in \Par(n)$ we write $\Sp(\lambda)_R$ for the corresponding Specht module over group algebra of $\Sym(n)$ over a commutative  ring $R$.  We write  $f_\lambda$ for the number of standard $\lambda$ tableaux, i.e., the dimension of $\Sp(\lambda)_\que$ over $\que$,    and $\chi^\lambda$ for the character of $\Sp(\lambda)_\que$.   Thus, for a subset $\Gamma$ of $\Par(n)$,  we have $|G(\Gamma)|=\sum_{\lambda\in \Gamma} f_\lambda^2$. Moreover, 
we  have $f_\lambda=f_{\lambda^*}$, $\lambda\in \Par(n)$.

\begin{remark}  For $\Delta\subseteq \Par(n)$ co-saturated  the dimension of $I_\que(\Delta)$ is $\sum_{\lambda\in \Delta} f_\lambda^2$. Indeed, for $\lambda\in \Delta$ we have $\que  \Sym(n) [\Sym(\lambda)]\subseteq I_\que(\Delta)$ and the left $\que \Sym(n)$-module $\que  \Sym(n) [\Sym(\lambda)]$ may be viewed as the induced module $\que\uparrow_{\Sym(\lambda)}^{\Sym(n)}$ and this has character of the form $\sum_{\mu\trianglerighteq \lambda}a_{\mu} \chi^\mu$,  for non-negative integers $a_{\mu}$, with $a_\lambda=1$, see e.g., Young's Rule, \cite[Section 7.3,  Corollary 1]{Fulton}.   Hence $\que \Sym(n) [\Sym(\lambda)]$ is a direct sum of modules of the form $\Sp(\mu)_\que$, for $\mu\in \Delta$. Hence, by co-saturation,  \\
$I_\que(\Delta)=\sum_{\lambda\in \Delta} \que \Sym(n) [\Sym(\lambda)]$ is isomorphic to $\bigoplus_{\lambda\in \Delta} \Sp(\lambda)_\que^{\oplus e_\lambda}$ for positive  integers $e_\lambda$.  Since $I_\que(\Delta)$ is an ideal, we have $e_\lambda=\dim \Sp(\lambda)_\que$, $\lambda\in \Delta$. 
Hence the $\que$-dimension of $I_\que(\Delta)$ is  $\sum_{\lambda\in \Delta} f_\lambda^2$.
\end{remark}

\begin{lemma} Let  $\Delta$ be a co-saturated  subset of $\Par(n)$. Then $I(\Delta)$ is a free $\zed$-module of rank $\sum_{\lambda\in \Delta} f_\lambda^2$.
\end{lemma}

\begin{proof} We have the natural isomorphism $\que \otimes_\zed  I(\Delta)\to I_\que(\Delta)$ and $I(\Delta)$ is torsion free and hence free and its rank must be the $\que$-dimension of $I_\que(\Delta)$ and this is $\sum_{\lambda\in \Delta} f_\lambda^2$.
\end{proof}

\q The group algebra  $R\Sym(n)$ has the $R$-algebra automorphism ${}^\dagger$ given on group elements by $\sigma^\dagger=\sgn(\sigma) \sigma$, where $\sgn(\sigma)$ is the sign of $\sigma\in \Sym(n)$. For $D\subseteq R\Sym(n)$ we set $D^\dagger=\{d^\dagger \vert d\in D\}$.

   \begin{proposition}   Let $\Gamma$ a co-saturated subset of $\Par(n)$ that has a unique minimal element and let $\Delta=\Par(n)\backslash \Gamma^*$. 
   
   (i) $R\Sym(n)=RG(\Gamma)\oplus I_R(\Delta)=RG(\Gamma)\oplus I_R(\Delta)^\dagger$. 
   
   (ii) $I_R(\Delta)$ (\resp $I_R(\Delta)^\dagger$)  is free,  as  an $R$-module,  of rank $\sum_{\lambda\in \Delta} f_\lambda^2$ and the natural map $R\otimes_\zed I_\zed(\Delta)\to I_R(\Delta)$ (\resp $R\otimes_\zed I_\zed(\Delta)^\dagger \to I_R(\Delta)^\dagger$)  is an isomorphism.
 \end{proposition}
 
 \begin{proof}  We note that if $\mu$ is the unique minimal element of $\Gamma$ then $\mu^*$ is the unique maximal element of $\Gamma^*=\Par(n)\backslash \Delta$.
 
 \q    The rank of $\zed G(\Gamma)$ is $|G(\Gamma)|=\sum_{\lambda\in \Gamma} f_\lambda^2=\sum_{\lambda\in \Gamma^*} f_\lambda^2$ and the rank of $I(\Delta)$ is $\sum_{\lambda\in \Delta}f_\lambda^2=\sum_{\lambda\in \Par(n)\backslash \Gamma^*}f_\lambda^2$.  Hence the rank of $\zed G(\Gamma)\bigoplus I(\Delta)$ is $\sum_{\lambda\in \Par(n)}f_\lambda^2=|\Sym(n)|$.  Thus the   natural map $\zed G(\Gamma)\bigoplus I_\zed(\Delta)\to \zed S(n)$ is, by Lemma 2.5,  a surjective homomorphism  between finitely generated torsion free abelian groups of the same rank and hence an isomorphism. Thus we have $\zed \Sym(n)= \zed G(\Gamma)\bigoplus I_\zed(\Delta)$.  Thus $I_\zed(\Delta)$ is a pure subgroup of $\zed S(n)$ and so the natural map $R\otimes_\zed \Sym(n)\to R \Sym(n)$ restricts to an  injective map $R\otimes_\zed I_\zed(\Delta)\to I_R(\Delta)$. However the map is surjective,  from the definitions,  and hence an isomorphism. This proves $R\Sym(n)=RG(\Gamma)\bigoplus I_R(\Delta)$  and hence, applying ${}^\dagger$ also $R\Sym(n)=RG(\Gamma) \bigoplus I_R(\Delta)^\dagger$.

  \q  We have shown that $R\otimes_\zed I(\Delta)\to I_R(\Delta)$ is an isomorphism. To get the statement for $I_R(\Delta)^\dagger$ apply the anti-automorphism ${}^\dagger$.
 \end{proof}

 \q  For $\lambda\in \Par(n)$ we have the $\Sym(n)$-set $\Sym(n)/\Sym(\lambda)$. We write $M(\lambda)_R$ for the corresponding permutation module for $R\Sym(n)$. This may be viewed as the induced module $R\Sym(n)\otimes_{R\Sym(\lambda)} R$, and the natural map $M(\lambda)_R\to R\Sym(n) [\Sym(\lambda)]$ is an $R\Sym(n)$-module isomorphism.
  If $\phi:R\to R'$ is a homomorphism of commutative rings we have the natural $R'\Sym(n)$-module isomorphism $R'\otimes_R M(\lambda)_R\to M(\lambda)_{R'}$.
 
\q   We say that an  $R \Sym(n)$-module $M$ is a Young permutation module if it is isomorphic to $\bigoplus_{\lambda\in \Par(n)} M(\lambda)_R^{\oplus u_\lambda}$ for some non-negative integers $u_\lambda$, $\lambda\in \Par(n)$. We say that  such a module $M$ has type $\Gamma$, for some $\Gamma\subseteq \Par(n)$, if $u_\lambda \neq 0$  if and only if  $\lambda\in \Gamma$.

  \q We have the $R$-linear map $\ep_R: R \Sym(n)\to R$ given on permutations by 
   $$\ep_R(\sigma)=
 \begin{cases} 1, & \hbox{ if } \sigma=1;\\
 0, & \hbox{ if } \sigma\neq 1
 \end{cases}
 $$
  and the symmetric, bilinear form on $R\Sym(n)$ given by $(f,f')=\ep_R(ff')$.  We have $(ff',f'')=(f,f'f'')$, for $f,f',f''\in R\Sym(n)$ and moreover the form induces an $R$-module  isomorphism 
  $$\eta_R: R \Sym(n)\to \Hom_R(R\Sym(n),R)$$
   given by 
  $\eta_R(x)(y)=(x,y)$.

  \begin{remark} Let $\Gamma$ be a co-saturated subset of $\Par(n)$  with unique minimal element and let  $\Delta=\Par(n)\backslash \Gamma^*$. Then $I_R(\Delta)$ is an $R$-module  direct summand of $R\Sym(n)$, by Proposition 2.8, and hence the map $\zeta_R: R\Sym(n)\to \Hom_R(I_R(\Delta),R)$, given by $\zeta_R(x)(y)=(x,y)$,  is surjective. 
  \end{remark}
  
  \q We shall write $\Ann_A(M)$ for the annihilator $\{a\in A \vert aM=0\}$ of a left module $M$ over a ring $A$.

 \begin{corollary} Let  $\Gamma$ be a co-saturated subset of $\Par(n)$ that has a unique minimal element  and let $M$ be a Young permutation module for $R\Sym(n)$  of type $\Gamma$.

 (i)  The restriction $\theta_R: R G(\Gamma)\to \End_R(M)$ of the representation \\
 $\rho_R:R\Sym(n)\to \End_R(M)$  is injective.

 (ii) The annihilator of $M$ is $I_R(\Delta)^\dagger$, where $\Delta=\Par(n)\backslash \Gamma^*$. 
 \end{corollary}

 \begin{proof}  
  \q Consider the annihilator $\Ann_R(M)$.  For $x\in R\Sym(n)$ we have $x\in \Ann_R(M)$ if  and only if $x R\Sym(n)[\Sym(\lambda)]=0$ for all $\lambda\in \Gamma$. By  \cite[4.6 Lemma]{James} we have $I(\Delta)^\dagger\subseteq \Ann_{R\Sym(n)}(M)$.  Furthermore, we have  $x I_R(\Gamma)=0$  if and only if $(xI_R(\Gamma),y)=0$ for all $y\in R\Sym(n)$, i.e., if and only if $(x,I_R(\Gamma) y)=0$ for all $y\in R\Sym(n)$, i.e.,  if and only $(x,I_R(\Gamma))=0$. Thus $\Ann_{R\Sym(n)}(M)=I_R(\Gamma)^\perp$.

 \q We shall call  $R$ good if $\theta_R: R G(\Gamma)\to \End_R(M)$ is injective and call $R$ bad otherwise.  
 
 \q Note that if $R$ is a subring of a commutative ring $R'$ and $R'$ is good then, by considering the action of $R'\Sym(n)$ on $R'\otimes_R M$, we get that $R$ is good. Note also that if $J$ is an ideal of $R$ and $R/J$ is good and $\sum_{\sigma\in G(\Gamma)} x_\sigma \sigma\in \Ker(\theta_R)$ then from,  the injectivity of $\theta_{R/J}:(R/J) G(\Gamma) \to R/J\otimes_R M$, we get that   $x_\sigma\in J$, for all $\sigma\in G(\Gamma)$.
 
 \q We start by showing that  Artinian rings are good.  So suppose $R$ is Artinian. We write $l(X)$ for the composition length of a finitely generated $R$-module $X$.  Thus  $l(R)$ denotes the composition length of the left regular module.

 \q The natural map $\zeta_R: R\Sym(n)\to \Hom_R(I_R(\Gamma),R)$ is surjective (by Remark  2.9 with $\Gamma$ replaced by $\Delta$)   and its kernel is $I_R(\Gamma)^\perp=\Ann_{R\Sym(n)}(M)$.  Hence $\zeta_R$  induces an
 $R$-module  isomorphism $R\Sym(n)/\Ann_{R\Sym(n)}(M)\to \Hom_R(I_R(\Gamma),R)$ and hence we have 
 $$l(R\Sym(n)/\Ann_{R\Sym(n)}(M))=l(\Hom_R(I_R(\Gamma),R)).$$
\q Now $R\Sym(n) = RG(\Delta)\oplus I_R(\Gamma)$, by Proposition 2.8,  so that 
\begin{align*}l(\Hom_R(I_R(\Gamma),R)=l(I_R(\Gamma))&= l (R\Sym(n)) - l(R) |G(\Delta)|.
\end{align*}
Thus we have 
$$l(R\Sym(n))-l(\Ann_{R\Sym(n)}(M))=l(R\Sym(n))-l(R)|G(\Delta)|$$
and so
$$l(\Ann_{R\Sym(n)}(M))=l(R)|G(\Delta)|$$
which, by Proposition 2.8 (ii), is $l(I_R(\Delta))=l(I_R(\Delta)^\dagger)$.  But now \\
$l(\Ann_{R\Sym(n)}(M))=l(I_R(\Delta)^\dagger)$ and  $I_R(\Delta)^\dagger\subseteq \Ann_{R\Sym(n)}(M)$. Hence  $\Ann_{R\Sym(n)}(M)=I_R(\Delta)^\dagger$.  Thus if $x\in \Ker(\theta_R)$ then $x\in I_R(\Delta)^\dagger$ so  $x\in RG(\Gamma) \bigcap I_R(\Delta)^\dagger$ and  so  $x=0$ by Proposition 2.8(i). 
 
 \q In particular we get that $R$ is good if it is a  field and hence also $R$ is good if it an integral domain (a subring of a field). 
 
 \q Suppose, for a contradiction, that there is a bad commutative ring $R$. Thus we have a non-zero element $\sum_{\sigma\in G(\Gamma)} x_\sigma \sigma\in \Ker(\theta_R)$. If $P$ is a prime ideal of $R$ then, since $R/P$ is good,  all $x_\sigma\in P$. Hence each $x_\sigma$ belongs to the intersection of all prime ideals, i.e., the nilradical.  We may replace $R$ by the subring generated by all $x_\sigma$. Since each $x_\sigma$ is nilpotent, this is  finitely generated as an abelian group.   For a positive integer $d$, the ring $R/dR$ is finite and hence good. Thus, for all $\sigma\in G(\Gamma)$, we have  $x_\sigma\in \bigcap_{d\in \en}  dR=0$. This completes the proof of (i). 
 
 \q As  noted above, we have $I_R(\Delta)^\dagger \subseteq \Ann_{R\Sym(n)}(M)$ so that (i) and Proposition 2.8(i)  give  (ii).
  \end{proof}

  \begin{example}  Let $n,r\geq 1$. Consider  the action of $\Sym(r)$ on $I(n,r)$ by composition of mappings, i.e., $\sigma i=i\circ \sigma^{-1}$, for $\sigma\in \Sym(r)$, $i\in I(n,r)$.  The permutation module $R I(n,r)$ for $R\Sym(r)$ may be viewed as the $r$-fold tensor product
  $V^{\otimes r}=V\otimes_R \otimes  \cdots \otimes_R V$
  where $V$ is a free  $R$-module  of rank $n$, with the action by place permutations.   The stabiliser of $i\in I(n,r)$ is the Young subgroup 
  $\Sym(\alpha)$, where $\alpha=(\alpha_1,\alpha_2,\ldots)$, with  $\alpha_1=|i^{-1}(1)|, \alpha_2=|i^{-1}(2)|$, and so on.  Thus $RI(n,r)$ is a Young permutation module of type $\Gamma=\Lambda^+(n,r)$, the set of  partitions of $r$ with at most $n$ parts.   Then $\Lambda^+(n,r)$ has unique minimal element $\beta$, where $\beta=(r)$, if $n=1$ and $\beta=(a,a,\ldots,a,b)$, if $n>1$, and $r=(n-1)a+b$, for some negative integers $a,b$ with $b\leq  n-1$.

  \q Now $\Delta=\Par(r) \backslash \Lambda^+(n,r)^*$, is the set of partitions $\mu$ of $r$ with first entry $\mu_1>n$.   By the Corollary 2.10,(ii), the annihilator $J$ of $R I(n,r)=V^{\otimes r}$ is $0$ if $n\geq r$ and is otherwise generated by $\sum_{\sigma\in \Sym(X)} \sgn(\sigma) \sigma$, where $X$ is a subset of $[1,r]$ of size $n+1$.  Moreover, $R\Sym(r)/J$ has basis $\sigma + J$, $\sigma\in G(\Gamma)$.  
  
  \q Now $\pi(s)\in G(\Gamma)$ if and only if $\sh(s)\in \Gamma$, i.e.,  if and only if $\sh(s^*)\in \Gamma^*$,  i.e.,  if and only if $\sh(s^*)_1\leq n$, i.e.,  if and only if $\lambda\in A(s^*)$ implies $\lambda_1\leq n$, i.e.,  if and only if $s$ has no descending subsequence of length $n+1$.  This is Theorem 1 of \cite{RSS}. 
   \end{example}

\q We make a slight extension to  Corollary 2.10.    If $\mu=(\mu_1,\mu_2,\ldots)$ is a partition of $n$ and $\mu_r=a+b$, for some $r\geq 1$, $a,b>0$ then we say that the partition whose parts are $\mu_1,\ldots,\mu_{r-1},a,b,\mu_{r+1},\ldots$ (in some order) is a \emph{simple refinement}  of $\mu$. For  $\lambda,\mu\in \Par(n)$ we say  that $\lambda$ is a  refinement of $\mu$ if  there is a sequence of partitions $\lambda=\tau(0),\tau(1),\ldots,\tau(m)=\mu$ such that $\tau(i)$ is a simple refinement of $\tau(i+1)$ for $0\leq i<m$.   For $\lambda,\mu\in \Par(n)$ we say that $\mu$ is a coarsening of  $\lambda$ if $\lambda$ is a refinement of $\mu$. (For more information see e.g., \cite[Chapter I, Section 6]{Mac}.)

\q Suppose that $\mu=(\mu_1,\mu_2,\ldots)\in \Par(n)$ and that $\lambda$ is a simple refinement of $\mu$ whose parts are $\mu_1,\ldots,\mu_{r-1},a,b,\mu_{r+1},\ldots$ (in some order) with $\mu_r=a+b$ with $a,b>0$.    Let $X_1,X_2,\ldots$ be a subsets of $[1,n]$ such that each $X_i$ has size $\mu_i$, $i\geq 1$. Let $X_r=Y\bigcup Z$, for subsets $Y,Z$ of sizes $a$ and $b$.  Then $\Sym(X_1)\times \cdots\times \Sym(X_{r-1}) \times \Sym(Y)\times \Sym(Z)\times \Sym(X_{r+1}) \times \cdots$ embeds in $\Sym(X_1)\times\cdots\times \Sym(X_r) \times \cdots$ and so a conjugate of $\Sym(\lambda)$ is contained in a conjugate of $\Sym(\mu)$. Hence there is a surjection of $\Sym(n)$-sets from $\Sym(n)/\Sym(\lambda)$ to $\Sym(n)/\Sym(\mu)$. Hence $\Ann_{R\Sym(n)}(M(\mu)_R)\subseteq \Ann_{R\Sym(n)}(M(\lambda)_R)$.

  \q For $\Gamma\subseteq \Par(n)$ we write $\hat\Gamma$ for the set of all $\mu \in \Par(n)$ such that $\mu$ is a coarsening of some $\lambda\in \Gamma$.
  
  \begin{corollary} Let $\Gamma$  be a subset of $\Par(n)$ containing a unique minimal element  and let $M$ be a Young permutation module for $R\Sym(n)$  of type $\Gamma$. If $\hat\Gamma$ is co-saturated then $\Ann_{R\Sym(n)}=I_R(\Delta)^\dagger$, where $\Delta=(\Par(n)\backslash \hat\Gamma)^*$ and the elements $\sigma + \Ann_{R\Sym(n)}(M)$, $\sigma \in G(\hat\Gamma)$ form an $R$-basis of \\
  $R\Sym(n)/\Ann_{R\Sym(n)}(M)$. 
  \end{corollary}

  \begin{proof}  We put $\hat M= \bigoplus_{\lambda\in \hat\Gamma} M(\lambda)_R$. Then $\Ann_{R\Sym(n)}(M)=\Ann_{R\Sym(n)}(\hat M)$. The result follows from Corollary 2.11.
  \end{proof}

  \begin{example}  We consider the permutation $\Sym(n)$ module $V$ with $R$-basis $v_1,\ldots,v_n$, on which $\Sym(n)$ acts by $\sigma v_i=v_{\sigma(i)}$, $\sigma\in \Sym(n)$, $1\leq i\leq n$.  Then $V^{\otimes r}$ is an $R\Sym(n)$-module. It is the permutation module on the set $I(n,r)$ with action $\sigma\cdot i=\sigma\circ i$.  The point stabiliser  of $i\in I(n,r)$ is conjugate to $\Sym(a,1^b)$, where $b$ is the size of the image of $i$. Thus $V^{\otimes r}$ is a Young module of type
  $$\Gamma=\{(a,1^b) \vert a\geq n-r\}$$
  and 
  $$\hat\Gamma=\{\lambda=(\lambda_1,\lambda_2,\ldots) \in \Par(n)\vert \lambda_1\geq n-r\}.$$
   This is a co-saturated set. 
  Moreover $\Gamma$ (and hence $\hat \Gamma$) has unique minimal element $(n-r,1^r)$ (if $r <n$, and otherwise $\Gamma=\Par(n)$).  Hence we may apply Corollary 2.12.
 We write $\len(\mu)$ for the length of a partition $\mu\in \Par(n)$. We have 
  $$\Delta=(\Par(n)\backslash \hat\Gamma)^*=\{\mu \in \Par(n)\vert \len(\mu)<n-r\}.$$
  
  The kernel of the action of $R\Sym(n)$ on $I(n,r)$ is $I_R(\Delta)^\dagger$ and the elements 
  
 $\sigma + I_R(\Delta)^\dagger$, $\sigma\in G(\hat \Gamma)$,  form an $R$-basis of $R\Sym(n)/I_R(\Delta)^\dagger$.   Explicitly, the kernel of the action of $R\Sym(n)$ on $V^{\otimes r}=RI(n,r)$ is  the ideal $I_R(\Delta)^\dagger$, generated by the elements $\sum_{\sigma\in \Sym(\mu)} \sgn(\sigma) \sigma$, as $\mu$ ranges over all partitions of length less than $n-r$ and $G(\hat \Gamma)$ is the set of $\pi(s)$, such that $s\in S(n)$ contains an increasing sequence of length $n-r$. 
 
 \q This is essentially \cite[Theorem 1]{DoDS}.
  \end{example}

\begin{remark}  It would be interesting to know whether the main results of this section, in particular Corollary 2.12, hold for all co-saturated subsets $\Gamma$ of $\Par(n)$ (without the restriction that $\Gamma$ has a unique minimal element). In fact all goes through without change provided that  $M(\Delta)=\zed G(P\backslash \Delta)+ I(\Delta)$ is $\zed \Sym(n)$, equivalently $M'(\Delta)=\zed S(n)$  (as in Lemma 2.5). 
\end{remark}

  \begin{example} For $n\leq 5$ the set $\Par(n)$ is totally ordered and so that Lemma 2.5 holds   for all co-saturated subsets $\Delta$.  For $n=6$ and $\Delta$ a co-saturated subset of $P=\Par(6)$ the complement $P\backslash \Delta$ has a unique maximal element unless $\Delta$ is either   $\Delta=\{6,(5,1),(4,2)\}$ or \\
  $\{6,(5,1),(4,2),(4,1,1),(3,3),(3,2,1)\}$.  We check that the desired property  holds in these cases too.

(i)  We consider first $\Delta=\{6,(5,1),(4,2)\}$.  Suppose, for a contradiction  $M'(\Delta)\neq \zed S(6)$ and  let $s$ be  the largest element of $S(6)$ not in $M'(\Delta)$. Then $A(s)\bigcap \Delta=\emptyset$ and $\sh(s)\in \Delta$.  This implies that $\sh(s)=(4,2)$ but $(4,2)\not\in A(s)$. There are two possibilities (details omitted).  We can have $s=412563$ or $236145$. However, we have $\pi(412563)^{-1}=\pi(236145)$ and $M(\Delta)$ is invariant under the ring automorphism of $\zed \Sym(6)$  taking $\sigma\in \Sym(6)$ to $\sigma^{-1}$. Hence $412563\in M'(\Delta)$ if and only if  $236145\in M'(\Delta)$.  Thus we may assume $s=412563$.  Now we take $H=\Sym\{1,5\}\times \Sym\{2,3,4,6\}$.  Then we have $s[H]\in  M'(\Delta)$. However 
$$s[H]=s+t+y$$
where $t=412365$ and $y$ is a sum of terms $u\in S(6)$ with $u>s$. Hence $s+t\in M'(\Delta)$ and hence $t\not\in M'(\Delta)$. 

\q However, we make take now $H=\Sym\{1,6\}\times \Sym\{2,3,4,5\}$ and we get $t[H]\in M'(\Delta)$ but 
$$t[H]=t+z$$
where $z$ is a sum of elements $v\in  S(6)$ with $v>s$. Hence $t\in M'(\Delta)$, a contradiction.

\medskip
        
        (ii) We consider $\Delta=  \{6,(5,1),(4,2),(4,1,1),(3,3),(3,2,1)\}$.  Suppose, for a contradiction that $M'(\Delta)\neq \zed S(6)$ and  let $s$ be  the largest element of $S(6)$ not in $M'(\Delta)$.  Then $\sh(s)=(3,2,1)$. However this implies $(3,2,1)\in A(s)$ (details  omitted)  and so $A(s)\bigcap \Delta\neq \emptyset$, a contradiction. 
        
\q Hence the desired   property $\zed G(P\backslash \Delta)+ I(\Delta)=\zed \Sym(6)$ holds for all co-saturated subsets $\Delta$ of $\Par(6)$. 
\end{example}

\bs\bs\bs\bs

\section{Complements: Cell Structure, Annihilators, Generators and Relations}

\bs\bs

\bf Cell Structure

\bs

\rm \q We first recall the notion of a cellular algebra due to Graham and Lehrer, \cite{GL}. (We have made some minor notational changes, the most serious of which is the reversal of the partial order from the definition given in \cite{GL}.)

\begin{definition} Let $A$ be an algebra over a commutative ring $R$. A cell datum $(\Lambda,N,C,{}^*)$ for $A$ consists of the following.

(C1) A partially ordered set $\Lambda$ and, for each $\lambda\in \Lambda$, a finite set $N(\lambda)$ and an injective map $C:\coprod_{\lambda\in \Lambda} N(\lambda) \times N(\lambda) \to A$ with image an $R$-basis of $A$.

(C2) For  $\lambda\in \Lambda$ and $t,u\in N(\lambda)$ we write $C(t,u)=C^\lambda_{t,u}$. Then ${}^*$ is an  $R$-linear involutory anti-automorphism  of the ring $A$ such that $(C^\lambda_{t,u})^*=C^\lambda_{u,t}$.

(C3) If $\lambda\in \Lambda$ and $t,u\in N(\lambda)$ then, for any element $a\in A$, we have 
$$aC^\lambda_{t,u}\equiv \sum_{t'\in N(\lambda)} r_a(t',t) C^\lambda_{t',u} \hskip 20pt ({\rm mod} \  A(>\lambda))$$
where $r_a(t',t)$ is independent of $u$ and where $A(>\lambda)$ is the $R$-submodule of $A$ generated by $\{C^\mu_{t'',u''} \vert \mu\in \Lambda, \mu >\lambda {\rm \ and \   } t'',u''\in N(\mu)\}$. 
\end{definition}

\q We also call a cell datum for $A$ a cell structure on $A$. We suppose that $A$ is an $R$-algebra with a cell datum   $(\Lambda,N,C,{}^*)$. For a  co-saturated subset $\Phi$ of $\Lambda$  we have the corresponding cell idea $A(\Phi)$ and, for $\lambda\in \Lambda$, we have the cell module $W_A(\lambda)$, as in \cite{GL}.    If $\phi:R\to R'$ is a homomorphism of commutative rings then one obtains, by base change,  a cell datum on  $A_{R'}=R'\otimes_R A$. The inclusion $A(\Phi)\to A$ induces an isomorphism $R'\otimes_R A(\Phi) \to 
A'(\Phi)$ and we have the natural isomorphism of cell modules  $R'\otimes_R W_A(\lambda)\to W_{A'}(\lambda)$, $\lambda\in \Lambda$.

\q We specialize to group algebras of symmetric groups.  For a subset $\Gamma$ of $\Par(n)$ we define $B(\Gamma)_\que$ to be the annihilator of  $\bigoplus_{\lambda\in \Par(n)\backslash \Gamma} \Sp(\lambda)_\que$.  Equivalently $B(\Gamma)_\que$ is the largest ideal which, as a left $\que \Sym(n)$-module, is a direct sum of copies of $\Sp(\lambda)_\que$, $\lambda\in \Gamma$. 
We write $B(\Gamma)_\zed$ for $B(\Gamma)_\que \bigcap \zed \Sym(n)$.   Then $B(\Gamma)_\zed$ is a pure submodule of $\zed \Sym(n)$ so that, for an arbirary commutative ring $R$, the natural map $R\otimes_\zed B(\Gamma)_\zed\to R \Sym(n)$ is injective. We write $B(\Gamma)_R$ for its image.    For $\Gamma\subseteq \Par(n)$ we have $B(\Gamma^*)_R= B(\Gamma)_R^\dagger$.

\q Recall that, on $A=\zed\Sym(n)$, there is a cellular structure $(\Lambda,N,C,{}^*)$, with $\Lambda=\Par(n)$ with the dominance partial order and such that the cell module $W_{A_\que}(\lambda)$ for $A_\que=\que \Sym(n)$,  has character $\chi^\lambda$, $\lambda\in \Par(n)$, for example by \cite[proof of (5.7) Lemma]{GL}.   We shall say that such a cell structure on $\zed \Sym(n)$ is  regular.

\begin{remark} Suppose    $(\Lambda,N,C,{}^*)$ is a regular cell structure on \\
$A=\zed\Sym(n)$.  Let $\Gamma$ be a co-saturated subset of $\Par(n)$. Then, for any commutative ring $R$, we have $A(\Gamma)_R=B(\Gamma)_R$, in particular the cell ideals of $R\Sym(n)$ do not depend on the choice of  (regular) cell structure. 

\q To see this we note that $A(\Gamma)_\que$ is a direct sum of modules $\Sp(\lambda)_\que$, $\lambda\in\Gamma$ and $\que \Sym(n)/A(\Gamma)_\que$ is a direct sum of modules $\Sp(\lambda)_\que$, $\lambda\in \Par(n)\backslash \Gamma$. It follows that $A(\Gamma)_\que=B(\Gamma)_\que$. Also $A(\Gamma)_\zed$ annihilates $\bigoplus_{\lambda\in \Par(n)\backslash \Gamma} \Sp(\lambda)_\que$ and hence $A(\Gamma)_\zed \subseteq B(\Gamma)_\zed$. But $A(\Gamma)_\zed$ and $B(\Gamma)_\zed$ are torsion free of the same rank and $\zed \Sym(n)/A(\Gamma)_\zed$ and $\zed \Sym(n)/B(\Gamma)_\zed$ are torsion free. Hence we must have $A(\Gamma)_\zed=B(\Gamma)_\zed$.  The natural maps \\
$R\otimes_\zed A(\Gamma)_\zed\to A(\Gamma)_R$ and $R\otimes_\zed B(\Gamma)_\zed\to B(\Gamma)_R$ are isomorphism. The result follows.
\end{remark}

\q We fix a regular cell structure on $\zed \Sym(n)$.

\begin{remark}Let $\Gamma$ be a saturated subset of $\Par(n)$ and let $X$ be a direct sum of modules $\Sp(\lambda)_\que$, $\lambda\in \Gamma$,  each occurring at least once.   Then the annihilator of $X$ in $\zed \Sym(n)$ is the cell ideal $\zed \Sym(n)(\Phi)$, where $\Phi$ is the complement of $\Gamma$ in $\Par(n)$. 

\q Let $I$ be the annihilator of $X$ in $\zed \Sym(n)$.   The inclusion of $I$ in $\zed \Sym(n)$ induces an isomorphism $\que\otimes_\zed I\to \que \Sym(n) (\Phi)$. Hence the rank of $I$ is $\sum_{\lambda\in \Phi} f_\lambda^2$. Moreover, $I$ contains the cell ideal $\zed \Sym(n)(\Phi)$ which also has rank $\sum_{\lambda\in \Phi} f_\lambda^2$. Moreover $\zed \Sym(n)/I$ is torsion free and hence \\
$I=\zed \Sym(n)(\Phi)$. 
\end{remark}

\q Let $R$ be a commutative ring. If $Y$ is a $\zed \Sym(n)$-module we write $Y_R$ for the $R\Sym(n)$-module  $R\otimes_\zed Y$.

\begin{proposition} Let $\Gamma$ be a subset of $\Par(n)$ and let   $M$ be a finite  a direct sum of modules $M(\lambda)_\zed$, $\lambda\in \Gamma$, with each $M(\lambda)_\zed$ occurring at least once. If   ${\hat \Gamma}$ is co-saturated then the inclusion of  $\Ann_{\zed \Sym(n)}$ in $\zed\Sym(n)$ induces an isomorphism   
$R\otimes_\zed \Ann_{\zed \Sym(n)}(M) \to \Ann_{R\Sym(n)}(M_R)$.

\end{proposition}

\begin{proof}    Let $I$ be the annihilator of $M_\zed$ in $\zed \Sym(n)$.   Then   $I$ is a pure submodule of $\zed \Sym(n)$ so we identify $I_R$ with an ideal  of $R\Sym(n)$.   Certainly $I_R\subseteq \Ann_{R\Sym(n)}(M_R)$ so it suffices to show the reverse inclusion.

\q We consider an element $x=\sum_{\sigma\in \Sym(n)} x_\sigma  \sigma$ of $ \Ann_{R\Sym(n)}(M_R)$. Let $R'$ be the subring of $R$ generated by all $x_\sigma$. Then $R'$ is Noetherian. Moreover $x\in \Ann_{R'\Sym(n)}(M_{R'})$.   If the natural map $R'\otimes_\zed I\to R'\Sym(n)$ has image $\Ann_{R'\Sym(n)}(M_{R'})$ then $x$ belongs to the image of the map $R'\otimes_\zed I\to R'\Sym(n)$  and hence to the image of  $R\otimes_\zed I\to R\Sym(n)$. Thus we may assume that $R$ is Noetherian. 

\q Now $X=\Ann_{R\Sym(n)}(M_R)/I_R$ is a finitely generated $R$-module. 
 For a field $F$ the  inclusion of  $\Ann_{\zed \Sym(n)}(M_\zed)$ in $\zed\Sym(n)$ induces an isomorphism   $F\otimes_\zed \Ann_{\zed \Sym(n)}(M) \to \Ann_{F\Sym(n)}(M_F)$, by \cite[Proposition 5.9 (ii)]{DonkinPermMods}.  It follows that for every homomorphism from $R$ into a field $F$ we have $F\otimes_R X=0$. Hence $X=0$, i.e., $I_R=\Ann_{R\Sym(n)}(M_R)$.   \end{proof}
  
  \bs\bs

  \bf Annihilators of Young Permutation Modules
  
\bs
\rm

  \q   Our main purpose now is   the following. The result was obtained by elementary arguments  in Section 2 as Corollary 2.10(ii)  in the special case in which $\hat \Gamma$ has a unique minimal element.

  \begin{theorem} Let $\Gamma$ be a subset of $\Par(n)$ such that $\hat \Gamma$ is co-saturated  and let $\Delta=\Par(n)\backslash \hat \Gamma^*$.  Then $I_R(\Delta)^\dagger$ is the annihilator of any  Young permutation for $R\Sym(n)$ of type $\Gamma$.
   \end{theorem}

\q We argue by reduction to the case $R=F$, an algebraically closed field. To handle that case we make a digression into  the polynomial representation theory of the general linear group, in the setting due to  J. A. Green,  \cite{EGS}.

\q   For $F$-vector spaces $V,W$ we write simply  $V\otimes W$ for $V\otimes_F W$. 

\q Let $n$ be a positive integer and let $G(n)=\GL_n(F)$. We regard $G(n)$ as an algebraic group over $F$. For $1\leq i,j\leq n$ we have the coefficient function $c_{ij}$: for $g\in G(n)$, $c_{ij}(g)$ is its $(i,j)$-entry.   Then  $A(n)$ denotes the algebra of functions on $G(n)$ generated by all $c_{ij}$. The algebra  $A(n)$ is the polynomial algebra in the $c_{ij}$.

  \q  Now $A(n)$ has the structure of a bialgebra with structure maps \\
  $\de:A(n)\to A(n)\otimes A(n)$, $\ep: A(n)\to F$, satisfying
  $$\de(c_{ij})=\sum_{r=1}^n c_{ir}\otimes c_{rj} \hbox{ and } \ep(c_{ij})=\de_{ij}$$
  for $1\leq i,j\leq n$.  The $F$-algebra  $A(n)$ is  graded,  with each $c_{ij}$ having degree $1$,  and we have the decomposition into homogeneous components \\
  $A(n)=\bigoplus_{r\geq 0} A(n,r)$. The space $A(n,r)$ is a subcoalgebra of $A(n)$, for $r\geq 0$.  For $i,j\in I(n,r)$ we write $c_{ij}$ for the element $c_{i_1j_1}\ldots c_{i_rj_r}$. 
  Then $A(n,r)$ has the  $F$-basis $c_{ij}$, $i,j\in I(n,r)$.   The dual algebra $S(n,r)$ of $A(n,r)$ is the Schur algebra and has the basis $\xi_{ij}$, dual to the basis $c_{ij}$, $i,j\in I(n,r)$. 
  
  \q We write $\Lambda(n,r)$ for the set of sequences $\alpha=(\alpha_1,\ldots,\alpha_n)$ of non-negative integers such that $\alpha_1+\cdots+\alpha_n=r$. We write $\Lambda^+(n,r)$ for the set of $\alpha\in \Lambda(n,r)$ such that $\alpha_1\geq  \cdots \geq \alpha_n$ and identify $\Lambda^+(n,r)$ with the set of partitions of $r$ with at most $n$ parts, as in Example 2.11.

 \q The content $i\in I(n,r)$ is the element $\alpha=(\alpha_1,\ldots,\alpha_n)$ of $\Lambda(n,r)$, with first entry  $\alpha_1$  the number of $1$'s occurring in $i$, with second entry  $\alpha_2$   the number of $2$'s occurring in $i$, and so on. For $\alpha\in \Lambda(n,r)$ the element $\xi_\alpha$ is defined to be  $\xi_{ii}$, where $i$ is any element of $I(n,r)$ with content $\alpha$.    Let $\alpha\in \Lambda(n,r)$. A left (\resp right) $S(n,r)$-module $V$ has $\alpha$ weight space $V^\alpha=\xi_\alpha V$ (\resp ${}^\alpha V=V\xi_\alpha$).

  \q  We write $\Mod(G(n))$ for the category of rational $G(n)$-modules and \\
 $\mod(G(n))$ for the category of finite dimensional rational $G(n)$-modules.  We write $X^+(n)$ for the set of elements $\lambda=(\lambda_1,\ldots,\lambda_n)\in \zed^n$ such that $\lambda_1\geq\cdots \geq \lambda_n$.  For each $\lambda\in X^+(n)$ we have an irreducible rational $G(n)$-module $L(\lambda)$ with unique highest weight $\lambda$.  Let $\pi$ be a subset of $X^+(n)$. We say that $V\in \Mod(G(n))$ belongs to $\pi$ if each composition factor of $V$ has the form $L(\lambda)$, for some $\lambda\in\pi$.  For $V\in \Mod(G(n))$ there is,  among all submodules belonging to $\pi$,  a unique maximal one and we denote this $O_\pi(V)$.  We have a  left exact functor $O_\pi: \Mod(G(n))\to \Mod(G(n))$: if $f:V\to V'$ is a morphism then $O_\pi(f):O_\pi(V)\to O_\pi(V')$ is the restriction of $f$.  We can apply this to the coordinate algebra $F[G(n)]$, regarded as the left regular module. Then $O_\pi(F[G(n)])$ is a subcoaglebra of $F[G(n)]$. 
 
  \q For each $\lambda\in X^+(n)$ we have the induced module $\nabla(\lambda)$ with highest weight $\lambda$.    For $r\geq 0$ and $\lambda\in \Lambda^+(n,r)$ the dimension of $\nabla(\lambda)$ is the number of standard tableaux of shape $\lambda$, see e.g.,  \cite[(4.5a) Basis Theorem and (4.8f) Theorem]{EGS}.

\q  A good filtration of  $V\in \mod(G(n))$ is a filtration whose sections are induced modules and $V$ is called a tilting module if both $V$ and the dual module $V^*=\Hom_F(V,F)$ admit good filtrations.   

   \q We write $E$ for the natural $G(n)$-module.  For a finite string of non-negative integers  $\alpha=(\alpha_1,\alpha_2,\ldots)$ we write $\tbw^\alpha E$ for the tensor product 
 $\tbw^{\alpha_1}E \otimes \tbw^{\alpha_2} E\otimes \cdots$.     For $\lambda\in \Lambda^+(n)$ the module 
  $\tbw^{\lambda^*} E$ is a tilting module with unique highest weight $\lambda$, \cite[(3.4) Lemma (ii)]{DTilt}.
 
 \q Suppose that $\pi$ is  finite saturated subset of $\Lambda^+(n,r)$. We write $A(\pi)$ for $O_\pi(F[G])$.    Now  the dimension of $A(\pi)$ is equal to $\sum_{\lambda\in \pi} (\dim \nabla(\lambda))^2$, by \cite[(2.2c)]{ DSchurOne} and hence: 
 
 \bs

(1) \sl $\dim A(\pi)$   is equal to the number pairs of standard tableaux with common shape $\lambda$,  for some $\lambda\in \pi$.
\rm

 \bs
 
 \q  Let $V\in \mod(G(n))$ with basis  $v_1,\ldots,v_m$. Then we have the coefficient functions $f_{ij}\in F[G(n)]$, $1\leq i,j\leq m$, defined by the equations 
 $$gv_i=\sum_{j=1}^m f_{ji}(g)v_j$$
 for $g\in G(n)$. 
 The coefficient space $\cf(V)$ is the span of the elements $f_{ij}$, $1\leq i,j\leq m$. (It is independent of the choice of the basis.) It is a sub-coalgebra of $F[G(n)]$ and if  $V$ belongs to $\pi$ then  $\cf(V)\subseteq A(\pi)$.    We have  $A(\pi)\subseteq A(n,r)$ and,   by  \cite[Section 1, Theorem and Remark]{DInvSevMat}: 

\bs

(2) $A(\pi)=\sum_{\lambda\in \pi} \cf(\tbw^{\lambda^*} E)$.

\bs\rm

\q We shall need to know that $A(\pi)$ is spanned by bideterminants.    There are several treatments of  bideterminants  available (see especially \cite[Theorems 1,3]{DRS},   \cite[Theorem 3.2]{DEP}).  Our approach here is by taking $q=1$ in \cite[Section 1.3]{q-Donk}.
 
  \q For $i,j\in I(n,r)$ we define 
  $$(i:j)=\sum_{\sigma\in \Sym(r)} \sgn(\sigma)  c_{i,j\sigma}.$$
  
   More generally, for $i,j\in I(n,r)$, $\lambda\in \Par(r)$,  we define
  $$(i:j)^\lambda=\sum_{\sigma\in \Sym(\lambda)} \sgn(\sigma) c_{i,j\sigma}$$
  so that $(i:j)=(i:j)^{(r)}$.

   \q  We have the diagram $[\lambda]$ of $\lambda$, as in \cite[Section 4.2]{EGS}. By a $\lambda$-tableau we mean a map $T:[\lambda]\to \{1,2,\ldots,n\}$.  We write $\Tab(\lambda)$ for the set of $\lambda$-tableaux. We have the rows $T^i\in I(n,\lambda_i)$, $1\leq i\leq l$, of a $\lambda$-tableau $T$, where $l$ is the length of $\lambda$.  Thus $T^1=(T(1,1),\ldots,T(1,\lambda_1))$, $T^2=(T(2,1),\ldots,T(2,\lambda_2))$, and so on.   For $S,T\in \Tab(\lambda)$ we have the bideterminant
  $$(S:T)=(S^1:T^1)\ldots (S^l:T^l)\in A(n,r).$$

  \q For $T\in \Tab(\lambda)$ we have   the associated word $w(T)\in I(n,r)$ obtained by reading the entries from left to right (not right to left as in \cite[Chapter I, Section 9]{Mac}) and rows successively.  Then, for $S,T\in \Tab(\lambda)$, we have 
  $$(S:T)=(w(S):w(T))^\lambda.$$

      \q  The coefficient space of $\tbw^\lambda E$ is the $F$-span of the bideterminants $(S:T)$, $S,T\in \Tab(\lambda)$. See for example \cite[Section 1.3]{q-Donk} (taking $q=1$).

  \q Thus we have the following.
  
  \begin{lemma} Let $\pi$ be a saturated subset of $\Lambda^+(n,r)$. Then $A(\pi)$ is the $F$-span of all bideterminants $(S:T)$, $S,T\in \Tab(\lambda^*)$, $\lambda\in \pi$.
  \end{lemma}

    \q An element $T\in \Tab(\lambda)$ is called anti-standard if it entries are strictly  decreasing along  rows and weakly decreasing down columns.  For $T\in \Tab(\lambda)$ we write $T^*$ for  $\lambda^*$-tableau defined by $T^*(i,j)=T(j,i)$,  $(i,j)\in [\lambda]$. Then $T\in \Tab(\lambda)$ is anti-standard if and only if $T^*$ is standard in the usual sense (as in \cite{EGS}).    
  
  \q Now the bideterminants $(S:T)$ with $S,T$ anti-standard and of shape belonging to  $\pi^*$ are independent, by \cite[Theorem 1.3.4]{q-Donk} and hence the \\
  $(S:T)$, with $S,T$ anti-standard of shape belonging to $\pi^*$ span a subspace of dimension equal to that of  $A(\pi)$, by (1).    Thus  we have:

\bs

(3) \sl \, The bideterminants $(S:T)$ with $S,T$ anti-standard and of shape  belonging to  $\pi^*$ form a basis of   $A(\pi)$.
 \rm

 \bs
 \rm
 
 \q Now $A(n,r)$ is naturally an $S(n,r)$-bimodule with actions satisfying
 $$\xi\cdot c_{ij}=\sum_{p\in I(n,r)} \xi(c_{pj})  c_{ip}\q \hbox{ and } \q c_{ij}\cdot \xi=\sum_{p\in I(n,r)} \xi(c_{ip}) c_{pj}$$
 for $\xi\in S(n,r)$, $i,j\in I(n,r)$, see \cite[Section 4.4]{EGS}.  It follows that, for $\lambda\in \Par(r)$,  we also have
 $$\xi\cdot (i:j)^\lambda=\sum_{p\in I(n,r)} \xi(c_{pi})(i:p)^\lambda \hbox{ \ and \q } (i:j)^\lambda\cdot \xi =\sum_{p\in I(n,r)} \xi(c_{ip})(p:j)^\lambda.$$

 \q Thus we have $(i:j)\in {}^\alpha A(n,r)^\beta$,  where $i$ has content $\beta$ and $j$ has content $\alpha$ and, for $S,T\in \Tab(\lambda)$ we have $(S:T)\in {}^\alpha A(n,r)^\beta$, where $r=|\lambda|$, $w(S)$ has content $\beta$ and $w(T)$ has content $\alpha$. 
 

\q From (3) we have the following.

  \begin{lemma} Let $\pi$ be a saturated subset of $\Lambda^+(n,r)$ and $\alpha,\beta\in \Lambda(n,r)$. Then the dimension of ${}^\alpha A(\pi)^\beta$ is independent of the  (algebraically closed) field $F$. 
  \end{lemma}

  \q We now suppose  $r\leq n$ and transport some of this structure to the group algebra of the symmetric group.  We have the element $u=(1,2,\ldots,r)$ of $I(n,r)$ and the element $\omega=(1,1,\ldots,1,0,\ldots,0)$ of $\Lambda(n,r)$.

\q   We set $e=\xi_\omega$.  Then the algebra $eS(n,r)e$ is isomorphic to the group algebra of the symmetric group $F\Sym(r)$, via the isomorphism $\phi:F\Sym(r)\to eS(n,r)e$, satisfying $\phi(\sigma)=\xi_{u\sigma,u}$, $\sigma\in \Sym(r)$.  
  
  \q By restricting to  left and right $\omega$ weight spaces in Lemma 3.7, we get the following.
  
  \begin{lemma} Let $\pi$ be a saturated subset of $\Lambda^+(n,r)$. Then ${}^\omega A(\pi)^\omega$ is the span of all bideterminants $(S:T)$, with $S,T\in \Tab(\lambda)$ and $\lambda^*\in \pi$, and $w(S)=w(T)=\omega$.  Moreover such $(S:T)$ with $S,T$ anti-standard form a basis. In particular the dimension of ${}^\omega A(\pi)^\omega$ is independent of the  (algebraically closed) field $F$. 

  \end{lemma}

  \q We now relate $A(n,r)$ and the group algebra of the symmetric group. 
  
  \begin{lemma} The $F$-linear map $\psi: F\Sym(r)\to  {}^\omega A(n,r)^\omega$, given by  $\psi(\sigma)=c_{u,u\sigma}$, $\sigma\in \Sym(r)$, is an $F\Sym(r)$-bimodule isomorphism.
  \end{lemma}

\begin{proof} The space  ${}^\omega A(n,r)^\omega$ is an $F\Sym(r)$-subbimodule of $A(n,r)$ and $\psi$ is an $F$-linear isomorphism.  It thus suffices to show that $\psi$ commutes with the $\Sym(r)$-action. 

\q For $\sigma,\tau\in \Sym(r)$ we have 
\begin{align*}\sigma \psi(\tau)&=\xi_{u\sigma,u}\cdot c_{u,u\tau}=\sum_{p\in I(n,r)} \xi_{u\sigma,u}(c_{p,u\tau})  c_{u,p}\\
&=\sum_{p\in I(n,r)} \xi_{u\sigma\tau,u\tau}(c_{p,u\tau} ) c_{u,p}=c_{u,u\sigma\tau}\\
&=\psi(\sigma\tau).
\end{align*}
Hence $\psi$ respects the left action of $\Sym(r)$ and similarly the right action.
\end{proof}

\q We are now ready to give the proof of  Theorem 3.5. 

\bs

\it Step 1. \sl We may assume that $\Gamma$ is co-saturated. 

\bs
\rm

 \q   Let $M$ be a Young permutation module of type $\Gamma$ and let $\hat M=\bigoplus_{\lambda\in \hat \Gamma} M(\lambda)_R$.   Since   $M$ and $\hat M$ have the same annihilator we may replace $\Gamma$ by $\hat\Gamma$

\bs


   \it Step 2. \sl We may assume that $R$ is an algebraically closed   field.

   \bs\rm
   
 \q  By Proposition 3.4,  it is enough to take $R=\zed$. Suppose, $I(\Delta)^\dagger\neq \Ann_{\zed \Sym(n)}(M)$, where $M$ is a Young permutation module of type $\Gamma$.   Let $Q=\Ann_{\zed \Sym(n)}(M)/I(\Delta)^\dagger$.  Let $F$ be an algebraically closed field.
   
   \q We have  the commutative diagram 
   $$\begin{matrix}&F\otimes_\zed I(\Delta)^\dagger & \to & I_F(\Delta)^\dagger \\
   &\downarrow & & \downarrow \\
   &F\otimes_\zed \Ann_{\zed \Sym(n)}(M) & \to &\Ann_{F\Sym(n)}(M_F).
   \end{matrix}$$
   The top map  is surjective and, assuming the result for algebraically closed fields,    the right map  is surjective. The bottom map  is an isomorphism, by Proposition 3.4,  and so the left map is surjective. 
   
   \q But now we have the exact sequence 
   $$F\otimes_\zed I(\Delta)^\dagger \to F\otimes_\zed \Ann_{\zed \Sym(n)}(M)\to F\otimes_\zed Q\to 0.$$
   Hence $F  \otimes_\zed Q=0$, for all  algebraically closed fields   $F$ and  $Q$ is finitely generated so $Q=0$, i.e.,  $\Ann_{\zed\Sym(n)}(M)=I(\Delta)^\dagger$. 
   
   \bs

\it Step 3. \sl  Completion of proof.

\bs\rm

\q We may suppose  that $\Gamma$ is co-saturated and that  $R=F$, an algebraically closed field. 

\rm


\q For $x,y\in \Sym(r)$ and $\lambda\in \Par(r)$ we have 
\begin{align*} \psi(x[\Sym(\lambda)]^\dagger y)&=\sum_{\sigma\in \Sym(\lambda)} \sgn(\sigma) x c_{u,u\sigma} y=\sum_{\sigma\in \Sym(r)}\sgn(\sigma) c_{u,u x\sigma y}\\
&=\sum_{\sigma\in \Sym(r)}\sgn(\sigma) c_{u y^{-1},u x\sigma }=(u y^{-1}: ux)^\lambda
\end{align*}

\q But every bideterminant $(S:T)$, for $S,T\in \Tab(\lambda)$,  of left and right weight $\omega$, has the form  
$(u y^{-1}: ux)^\lambda$, for some $x,y\in \Sym(r)$. Hence ${}^\omega \cf(\tbw^\lambda E)^\omega  \subseteq \psi(I_F(\Delta)^\dagger)$, for $\lambda\in \pi^*$, by Lemma 3.8.    By  (2) above we have ${}^\omega A(\pi)^\omega \subseteq \psi(I_F(\Delta)^\dagger)$ and hence $\psi(I_F(\Delta)^\dagger)={}^\omega A(\pi)^\omega$. Thus  the dimension of $I_F(\Delta)^\dagger$ is the same for all algebraically closed fields $F$.

\q However   $\dim_\seee  I_\seee(\Delta)^\dagger$ is the ideal which, as a left $\seee \Sym(n)$-module, is a direct sum of copies of $\Sp(\lambda)_\seee$, $\lambda \in \Delta^*$,  and this is the annihilator of $\bigoplus_{\lambda\in \Gamma} M_\seee(\lambda)$.  Putting $M=\bigoplus_{\lambda\in \Gamma} M(\lambda)_\zed$, we have 
$$\dim_\seee I_\seee(\Delta)^\dagger = \dim \Ann_{\seee \Sym(n)}(M_\seee).$$
  But this is also the $F$-dimension of $\Ann_{F\Sym(n)}(M_F)$ and \\
   $I_F(\Delta)^\dagger\subseteq \Ann_{F\Sym(n)}(M_F)$ so we must have 
  $\Ann_{F\Sym(n)}(M_F)=I_F(\Delta)^\dagger$, as required. 

   \bs\bs

\bf Generators and Relations

\rm
\bs

   \q Let $R$ be a commutative ring and let $\Gamma$ be a subset of $\Par(n)$ such that $\hat \Gamma$ is saturated.  We shall give a presentation of the image ${\rm Im}(\rho)$ of  the representation $\rho:R\Sym(n)\to \End_{R\Sym(n)}(M)$, afforded by a Young permutation module $M$ over $R$ of type $\Gamma$.  This is done in \cite[Proposition 5.1]{BDMCanBases}  for the special cases of the  Young permutation modules   considered in Examples 2.11 and 2.13 above.

\q   Let $\tilde A_R$ be the free $R$-algebra on generators $\tilde e_\sigma$, $\sigma\in \Sym(n)$.   Let $\Delta$ a subset of $\Par(n)$ and let $H_R(\Delta)$ be the ideal of $\tilde A_R$ generated by the elements:
  \medskip

    (1) $\tilde e_\sigma \tilde e_{\tau} - \tilde e_{\sigma\tau}$, for all $\sigma,\tau\in \Sym(n)$;  and
  
  (2) $\sum_{\sigma\in \Sym(\lambda)} \sgn(\sigma) \tilde e_\sigma$, for all $\lambda\in \Delta$.

 \q  We set $Q_R(\Delta)=\tilde A_R/ H_R(\Delta)$ and  $e_\sigma=\tilde e_\sigma + H_R(\Delta)$, $\sigma\in \Sym(n)$.

  \begin{proposition} Let $\Gamma$ be a subset of $\Par(n)$ such that $\hat\Gamma$ is co-saturated.  We put $\Delta=(\Par(n)\backslash {\hat \Gamma})^*$.  Let $M$ be Young permutation module for $R\Sym(n)$ of type $\Gamma$. Let $\rho:R\Sym(n)\to \End_R(M)$ be the representation afforded by $M$.  Then we have an $R$-algebra isomorphism $f: Q_R(\Delta)\to {\rm Im} (\rho)$ taking $e_\sigma$ to $\rho(\sigma)$, $\sigma\in \Sym(n)$.
  \end{proposition}
  
  \begin{proof} Let $\tilde f: \tilde A_R \to {\rm Im}(\rho)$ be the $R$-algebra homomorphism taking $\tilde e_\sigma$ to $\rho(\sigma)$, $\sigma\in \Sym(n)$.   Then 
  $$\tilde f( \tilde e_\sigma \tilde e_\tau  -\tilde e_{\sigma\tau})=\rho(\sigma)\rho(\tau)-\rho(\sigma\tau)=0$$
 for $\sigma,\tau\in \Sym(n)$. Also for $\lambda\in \Delta$, we have 
 $$[\Sym(\lambda)]^\dagger=\sum_{\sigma\in \Sym(\lambda)} \sgn(\sigma)\sigma\in I_R(\Delta)^\dagger$$
 and, by Theorem 3.5, $I_R(\Delta)^\dagger=\Ker(\rho)$ 
so that 
 $$\tilde f(\sum_{\sigma\in \Sym(\lambda)} \sgn(\sigma) \tilde e_\sigma)=\sum_{\sigma\in \Sym(\lambda)} \sgn(\sigma) \rho(\sigma)=\rho(\sum_{\sigma\in \Sym(\lambda)} \sgn(\sigma)\sigma)=0.$$
 
 \q Thus $\tilde f$ induces an $R$-algebra homomorphism $f:Q_R(\Delta)\to {\rm Im}(\rho)$ satisfying $f(e_\sigma)=\rho(\sigma)$, $\sigma\in \Sym(n)$. 
 
 \q Certainly $f$ is surjective.  We now check injectivity. Let $J$ be ideal of $\tilde A_R$ generated by all $\tilde e_\sigma \tilde e_\tau -\tilde e_{\sigma\tau}$.  Then the $R$-algebra surjection  $\tilde A_R\to R\Sym(n)$, taking $\tilde e_\sigma$ to $\sigma$, induces an isomorphism $\phi:\tilde A_R/J\to R\Sym(n)$, and $\phi^{-1}(I_R(\Delta)^\dagger)= H_R(\Delta)/J$. 
 
 \q Suppose $\sum_{\sigma\in \Sym(n)} a_\sigma e_\sigma$ (with $a_\sigma\in R$,  $\sigma\in \Sym(n)$) is in the kernel of $f$.  Then $\sum_{\sigma\in \Sym(n)} a_\sigma \sigma \in \Ker(\rho)=I_R(\Delta)^\dagger$. 
Thus \\
 $\phi(\sum_{\sigma\in \Sym(n)} a_\sigma \tilde e_\sigma + J)\in I_R(\Delta)^\dagger$, i.e., $\sum_{\sigma\in \Sym(n)} a_\sigma \tilde e_\sigma + J \in H_R(\Delta)/J$  and hence $\sum_{\sigma\in \Sym(n)} a_\sigma \tilde e_\sigma \in H_R(\Delta)$ and so $\sum_{\sigma\in \Sym(n)} a_\sigma e_\sigma=0$.  Thus $f$ is injective and hence an isomorphism.
\end{proof}

\bs\bs\bs\bs



\begin{thebibliography}{99}










\bibitem{AM} {M. F. Atiyah and I. G. Macdonald,  Introduction to Commutative Algebra, Adison-Wesley, 1969.}












\bibitem{BDMCanBases}{Chris Bowman, Stephen Doty and Stuart Martin, \emph{Canonical bases and new applications of increasing and decreasing sequences to invariant theory}, J. Algebra, {\bf 659}, (2024), 23-42.}





















\bibitem{DEP}{C. DeConcini, David Eisenbud and C. Procesi, \emph{Young Diagrams and Determinantal Varieties}, Invent. Math,  {\bf 56}, (1980), 129-165.}














\bibitem{DSchurOne} {S. Donkin, \emph{On Schur Algebras and Related Algebras I}, J.  of    Algebra, {\bf 104},  310-328,  1986.}




\bibitem{DInvSevMat}{S. Donkin, \emph{Invariants of several matrices}, Invent. Math. {\bf 110}, (1992), 389-401.}

\bibitem{DTilt} {S. Donkin, \emph{On tilting modules for algebraic groups},  Math. Z. {\bf 212}, 39-60, 1993.}












\bibitem{q-Donk} {S. Donkin, \emph{The $q$-Schur algebra},  LMS Lecture Notes 253, Cambridge University Press 1998.}






\bibitem{DonkinPermMods}{S. Donkin, \emph{Double centralisers and annihilator ideals of Young permutation modules},  J. of Algebra, {\bf591}, (2022), 249-288.}





















\bibitem{DoDS} {Stephen Doty, \emph{Doubly stochastic matrices and Schur-Weyl duality for partition algebras}, Electron. J. Combin. {\bf 29} (2022), no. 4, Part.  No. 4.28, 18 pages.}






\bibitem{DRS}{P. Doubilet, G. C. Rota and J.Stein, \emph{On the Foundations of Combinatorial Theory}, Vol. IX, Studies in Applied Mathematics, {\bf 53}, (1974), 185-216.}








\bibitem{Fulton}{William Fulton, \emph{Young Tableaux}, London Math. Society Texts {\bf 35},  Cambridge University Press, 1997.}


\bibitem{Geck}{M.  Gecke, \emph{Kazhdan-Lusztig cells and the Murphy basis }, Proc. London Math. Soc.  (3) {\bf 93}, (2006),  635-665.}




\bibitem{GL}{J. J. Graham and G. I. Lehrer, \emph{Cellular algebras}, Invent. Math. {\bf 123},  (1996),    1-34.}







\bibitem{EGS}{J. A. Green,  \emph{Polynomial Representations of ${\rm GL}_n$,  Second Edition with an Appendix on Schenstead Correspondence and Littelmann Paths by K. Erdmann, J. A. Green and M. Schocker}, Lecture Notes in Mathematics {\bf 830}, Springer 2007.}


\bibitem{Greene}{Curtis   Greene, \emph{An extension of Schensted's Theorem}, Advances in Mathematics.  {\bf 14}, (1974),  254-265.}













\bibitem{James}  {G. D. James,  \emph{The Representation Theory of the Symmetric Groups}, Lecture Notes in Mathematics 682, Springer 1970.}
































\bibitem{KL}{D. Kazhdan and G. Lusztig,  \emph{Representations of  Coxeter groups and Hecke algebras},  Invent. Math.    {\bf 53},  (1979),  165-184.}





\bibitem{Mac}{I. G. Macdonald, \emph{Symmetric Functions and Hall Polynomials}, 2nd Ed., Oxford Mathematical Monographs, Oxford University Press 1998.}














\bibitem{Murphy} {G. Murphy, \emph{The representations of Hecke algebras of type $A_n$},  J. Algebras   {\bf 173}, (1995),  97-121.}






\bibitem{RSS}{K. N. Raghavan, Preena Samuel and K. V. Subrahmanyam, \emph{RSK bases and Kazhdan-Lusztig cells}, Ann. Inst. Fourier, Grenoble, {\bf 62}, 2 (2012), 525-569.}



\bibitem{Schensted} {C. Schensted, \emph{Longest increasing and decreasing subsequences},  Canad.  J. Math.   {\bf 13}, (1961),  179-191.}
 
 
 
 


\end{thebibliography}
\end{document}